\documentclass[a4paper,reqno]{amsart}

\usepackage{comment}

\usepackage[utf8]{inputenc}
\usepackage[english]{babel}
\usepackage[T1]{fontenc}
\usepackage{microtype} 

\usepackage{amsmath,amssymb,amsthm}

\usepackage{thmtools} 
\usepackage{mathtools} 
\usepackage{upgreek}
\usepackage{orcidlink}

\usepackage{hyperref}
\hypersetup{%
  colorlinks=true,
  linkcolor=blue,
  citecolor=green!60!black
}
\usepackage{cleveref} 

\usepackage{ifthen}

\usepackage{tikz} 
\usetikzlibrary{cd}
\usetikzlibrary{calc}

\usepackage{pifont} 

\usepackage{accents} 

\usepackage{enumitem}
\newlength{\mylistindent}
\setlist[itemize,1]{leftmargin=\mylistindent,itemsep=2pt,parsep=2pt,topsep=3pt}
\setlist[enumerate,1]{label=\textup{(\roman{*})},leftmargin=\mylistindent,topsep=3pt,itemsep=2pt,parsep=2pt}

\newlist{steps}{enumerate}{1}
\setlist[steps]{label=\textup{\textbf{Step~\arabic{*}.}}, ref=\textup{Step~\arabic{*}}, align=left, leftmargin=0pt, itemindent=*,labelindent=0pt, labelsep=3pt, itemsep=3pt, parsep=2pt,topsep=2pt}

\allowdisplaybreaks

\numberwithin{equation}{section}

\let\geq\geqslant
\let\leq\leqslant

\newcommand*{\cM}{\ensuremath{\mathcal{M}}}

\newcommand*{\N}{\mathbb{N}} 
\newcommand*{\R}{\mathbb{R}} 
\newcommand*{\C}{\mathbb{C}} 

\renewcommand*{\epsilon}{\varepsilon} 
\newcommand*{\dd}{\mathrm{d}}

\newcommand*{\dx}[1][x]{\mathop{\dd{}#1}}

\newcommand{\cond}[1][\,]{#1:#1}

\newcommand{\conC}{\mathrm{C}}
\newcommand{\Cc}[1][\infty]{\conC^{#1}_{\mathrm{c}}}
\newcommand{\Lp}[2][]{\mathrm{L}_{#2\ifthenelse{\equal{#1}{}}{}{,#1}}} 
\newcommand{\Lb}{\mathcal{L}_{\mathrm{b}}} 
\newcommand{\soboW}{\mathrm{W}}
\newcommand{\soboH}{\mathrm{H}}

\newcommand{\sH}{\soboH}
\DeclareMathOperator{\cvl}{\ast}

\DeclareMathAccent{\Circ}{\mathalpha}{operators}{"17}

\DeclareMathOperator{\dom}{dom}

\DeclareMathOperator{\ran}{ran}

\DeclareMathOperator{\conv}{co}

\renewcommand{\div}{\operatorname{div}}

\DeclareMathOperator{\grad}{grad}

\DeclareMathOperator{\curlz}{\mathring{\operatorname{\curl}}}
\DeclareMathOperator{\divz}{\mathring{\operatorname{\div}}}
\DeclareMathOperator{\gradz}{\mathring{\operatorname{\grad}}}

\renewcommand{\Re}{\operatorname{Re}}

\DeclareMathOperator{\curl}{curl}

\newcommand{\ball}{\mathrm{B}}

\DeclarePairedDelimiter{\norm}{\lVert}{\rVert}
\DeclarePairedDelimiter{\abs}{\vert}{\vert}

\DeclarePairedDelimiterX{\dset}[2]{\{}{\}}{#1\mathrel{:}\mathopen{} #2} 
\DeclarePairedDelimiterX{\scprod}[2]{\langle}{\rangle}{#1,#2}

\newcommand{\cpt}{\overset{\mathsf{cpt}}{\hookrightarrow}}
\newcommand{\weakto}{\rightharpoonup}

\newcommand{\Htopo}{\mathrm{H}}

\theoremstyle{plain}

\newtheorem{theorem}{Theorem}[section]
\newtheorem{lemma}[theorem]{Lemma}
\newtheorem{proposition}[theorem]{Proposition}
\newtheorem{corollary}[theorem]{Corollary}

\theoremstyle{definition}

\newtheorem{remark}[theorem]{Remark}
\newtheorem{example}[theorem]{Example}

\makeatletter
\providecommand{\proofNameStyle}{\bfseries}
\renewenvironment{proof}[1][\proofname]{\par
  \pushQED{\qed}%
  \normalfont \topsep6\p@\@plus6\p@\relax
  \trivlist
  \item[\hskip\labelsep\proofNameStyle
  #1\@addpunct{.}]\ignorespaces
}{%
  \popQED\endtrivlist\@endpefalse
}
\makeatother

\begin{document}

\title{Nonlinear Media via Nonlocal Homogenisation}

\author[A.~Buchinger]{Andreas Buchinger\,\orcidlink{0009-0004-4203-5874}}
\address{TU Hamburg, Institute of Mathematics,
  Am Schwarzenberg-Campus 3, 21073 Hamburg, Germany}
\email{andreas.buchinger@tuhh.de}

\author[N.~Skrepek]{Nathanael Skrepek\,\orcidlink{0000-0002-3096-4818}}
\address{University of Twente, P.O.\ Box 217, Department of Applied Mathematics, 7500 AE Enschede, The Netherlands}
\email{n.skrepek@utwente.nl}

\author[M.~Waurick]{Marcus Waurick\,\orcidlink{0000-0003-4498-3574}}
\address{TU Bergakademie Freiberg, Institute for Applied Analysis, Pr{\"u}ferstra{\ss}e 9,  09599, Freiberg, Germany}
\email{marcus.waurick@math.tu-freiberg.de}

\makeatletter
\hypersetup{
  pdftitle={\@title},
  pdfauthor={A. Buchinger, N. Skrepek, M. Waurick},
}
\makeatother

\date{\today}


\begin{abstract}
  We consider a nonlinear PDE describing a nonlinear electrostatic medium with nonlocal dielectricity. The existence proof for the corresponding equation is based on Schauder's theorem and a new compactness theorem for moving coefficients (``Helga's Theorem''). This technique uses insights from (operator-theoretic/topological) homogenisation theory. Surprisingly, even though monotonicity assumptions are neither used nor valid, the underlying domain is only required to be weak Lipschitz and no assumption on the derivatives of the nonlinearity is needed.
\end{abstract}

\maketitle

\section{Introduction}

In this short note, we present an example where topological and operator-theoretic insights in homogenisation can lead to existence results for nonlinear partial differential equations with only a minimum of regularity needed. Hence, even though possible with moderate effort, we refrained from generalising this to other equations or dimensions in order to focus on a rather elementary yet nontrivial example, where the advantages of the method are clearly visible.
 
Our example is the following static, nonlinear Maxwell problem. Let $\Omega\subseteq \R^3$ be open, bounded, weak Lipschitz so that its complement is connected. For given $f \in \Lp{2}(\Omega)$ and $g\in\Lp{2}(\Omega)^3$ consider finding suitably regular $E\in \Lp{2}(\Omega)^3$ such that
\begin{equation}\label{eq:main-problem-handwavy}
  \left\{
  \begin{alignedat}{3}
    -\div \varepsilon(E) &= f &\qquad& \text{in}\ \Omega,  \\
    \curl E &= g && \text{in}\ \Omega, \\
    \nu \times E &= 0 && \text{on}\ \partial\Omega,
  \end{alignedat}
  \right.
\end{equation}
where $\nu$ denotes the normal vector on the boundary. The homogeneous tangential boundary condition $\nu \times E = 0$ is used formally here.\footnote{For strong Lipschitz $\Omega$, the tangential trace $\nu \times E$ is well-defined, so that the condition $\nu \times E = 0$ is meaningful in the usual sense. For weak Lipschitz domains, the tangential trace is generally not available, and the boundary condition is understood in an appropriate weak sense.}
 The nonlinearity is given by
\[
  \varepsilon(E) \coloneq a E + F(E)\cvl E
  \quad\text{with}\quad
  F(E)(x)\coloneq F(E(x)),
\]
where $F\colon \C^3 \to \C$ is bounded and continuous, $\cvl$ denotes the convolution, and $a\in \Lb(\Lp{2}(\Omega)^3)$ is a bounded operator, e.g., induced by multiplication with a $\Lp{\infty}(\Omega)^{3\times 3}$-matrix satisfying usual positive definiteness requirements. Given a size condition of the lower bound of $a$, the bound of $F$, and the volume of $\Omega$, we prove the existence of $\Lp{2}(\Omega)^3$-solutions $E$ via Schauder's fixed point theorem. 

Problems similar to this form have been considered in the physics literature; see, e.g., \cite{Sh89} and \cite[Chapter 2]{Graffi1980}. Further, note that nonlinear material models, where, for instance, the electric polarisation is nonlinear in the $E$-field, are considered to model various effects in the behaviour of electromagnetic wave propagation; see, e.g., \cite{KZ12}. In the more mathematical literature, (electro-)static partial differential equations have been addressed; for instance, in \cite{Picard1990}. However, in that reference, the coefficients are assumed to be maximal monotone. Such a sign-condition is not assumed here. For time-dependent problems, we refer to \cite{DIW24,Sp19,SS22} in order to provide exemplary references to nonlinear Maxwell problems that require some (piecewise) differentiability) smoothness of the coefficients. Here, we indeed avoid these types of assumptions entirely. Note that, in particular, due to the presence of the $\div a E$-term, spatial differentiation is not a viable option. This impedes a fixed-point approach in higher regularity Sobolev spaces. 

Our approach will focus on finding fixed points of the mapping
\[
  G\mapsto E \text{ solving }
  \left\{
  \begin{alignedat}{2}
    -\div (a E+ F(G)*E) &= f &\qquad&\text{in} \ \Omega, \\
    \curl E &= g &&\text{in} \ \Omega, \\
    \nu \times E &= 0 &&\text{on}\  \partial\Omega.
  \end{alignedat}
  \right.
\]
In order to work with Schauder's fixed point theorem, some compactness needs to be established. Again, since the term $\div a E$ is present, $\sH^1$-regularity (and hence Rellich's selection theorem for compactness) cannot be expected to hold either. Thus, other compactness results need to be employed. The root of this option is the Picard--Weber--Weck selection theorem (see \Cref{thm:PWWst} below), which establishes compactness of the embedding for vector fields with $\Lp{2}$-divergence and $\Lp{2}$-curl at the same time. This selection theorem is valid for (weak) Lipschitz domains and hence circumvents the necessity of Gaffney's inequality for higher regularity of $\Omega$. Quite recently in \cite{W22_tnt}, based on both such a compactness result and operator-theoretical analysis of homogenisation problems for div-curl systems of the above type, a compactness theorem covering ``moving coefficients'' coined ``Helga's theorem'' could be proven (\cite[Theorem 9.1]{W22_tnt} and \Cref{thm:Helga} below). It is this compactness result that, together with the basic homogenisation result in \cite[Theorem 7.5]{W22_tnt} (see  \Cref{thm:homnonl} below; see also \cite{NM22} for a similar context), renders the application of Schauder's theorem possible. Of course, one drawback of the method presented here is that the uniqueness of solutions is guaranteed only by assuming extra conditions on the coefficients, which we will not follow up on here.

We mention in passing that the homogenisation method coined in \cite{W18_NHC} and further developed in \cite{W22_tnt} provides the precise framework for homogenisation of nonlocal equations in mathematical physics and has been motivated by nonlocal constitutive relations in \cite[Chapter 10]{K11} in the context of nonlocal response theory. The motivation to homogenise nonlocal equations is also drawn from similar problem set-ups considered in, e.g., the physics papers \cite{G16,C2015,FL17}.

The main mathematical body of the paper starts with some preliminaries on vector-analytic operators and well-posedness for elliptic-type problems in \Cref{sec:background}. In this section, we also provide the solution theory for the linear version of the above nonlinear equation. In \Cref{sec:homnonl}, we recall and further detail the theory of homogenisation for possibly nonlocal coefficients in contrast to just multiplication operators as in the classical literature going back to the work of Tartar and Murat; see, e.g., \cite{Tartar2009}. \Cref{sec:mr} contains the main result and \Cref{sec:pmr} is devoted to its respective proof.

Throughout the manuscript, scalar products are conjugate linear in the first and linear in the second component. Weak convergence is denoted by $\weakto$. Bounded linear operators from one Banach space $X$ into another $Y$ are denoted by $\Lb(X,Y)$.

\section{Background on Operators in Vector-Analysis}\label{sec:background}
The so-called FA-toolbox in \cite{PZ2020} is a general reference for the functional analytic background of the machinery we will unfold for the de Rham complex in the following. For a slow-paced introduction, we additionally refer, e.g., to \cite[Chapter 6]{STW22}.

More specifically, we let $\Omega\subseteq \R^3$ be open, and we define the standard vector-analytic operators from the de Rham complex in $\R^3$: We set the gradient and the curl as
\begin{align*}
   \grad &\colon \dom(\grad)\subseteq \Lp{2}(\Omega) \to \Lp{2}(\Omega)^3; \phi\mapsto \nabla \phi = (\partial_j \phi)_{j\in \{1,2,3\}}\text{, and}\\
   \curl & \colon \dom(\curl)\subseteq \Lp{2}(\Omega)^3 \to \Lp{2}(\Omega)^3; E\mapsto \nabla \times E,
\end{align*}
where the derivatives $\nabla$ and $\nabla\times$ are taken in the distributional sense, and $\dom(\grad)$ and $\dom(\curl)$ are the maximal $\Lp{2}$-domains so that $\nabla \phi\in \Lp{2}(\Omega)^3$ and $\nabla \times E\in \Lp{2}(\Omega)^3$ respectively. Moreover, we define the corresponding operators with homogeneous boundary conditions via
\[
    \gradz \coloneq \overline{\grad\restriction_{\Cc(\Omega)}}\quad\text{ and }\quad\curlz \coloneq \overline{\curl\restriction_{\Cc(\Omega)^3}}.
\]
Finally, we define the divergence
\[
   \div \coloneq -\gradz^*\text{ and }\divz \coloneq -\grad^*.
\]
Note that the circle on top of $\grad$, $\curl$, $\div$ indicates that the corresponding domain incorporates homogeneous boundary conditions in a weak sense. In particular, it serves as a shorthand for the conditions $f=0$, $\nu \times E = 0$, and $\nu \cdot H = 0$ on $\partial\Omega$ for $f \in \dom(\grad)$, $E\in\dom(\curl)$, and $H\in\dom(\div)$, respectively, even when the corresponding traces are not well-defined.

Since test functions are dense in $\Lp{2}$, all the operators just introduced are densely defined, and they are also closed, for the latter follows by mere inspection and definition of the distributional derivatives.
Furthermore, it readily follows that $\curlz^*=\curl$.
As closed linear operators, the respective domains endowed with their respective graph scalar products are Hilbert spaces in their own right, and we shall denote these by
\[
   \sH^1(\Omega), \sH^1_0(\Omega), \sH(\curl,\Omega), \sH(\curlz, \Omega), \sH(\div,\Omega),\ \text{and}\ \sH(\divz,\Omega).
\]

In particular we can now formulate \eqref{eq:main-problem-handwavy} more precisely: For given $f \in \Lp{2}(\Omega)$ and $g \in \ran(\curlz)$ we aim to find $E \in \sH(\curlz,\Omega)$ such that $\varepsilon(E) \in \sH(\div,\Omega)$ and
\begin{equation*}
  \left\{
  \begin{alignedat}{3}
    -\div \varepsilon(E) &= f &\qquad& \text{in}\ \Omega, \\
    \curl E &= g && \text{in}\ \Omega.
  \end{alignedat}
  \right.
\end{equation*}
Note that the homogeneous tangential boundary condition is hidden in $E \in \sH(\curlz,\Omega)$.

A fundamental question concerning these operators is when these spaces are compactly embedded into $\Lp{2}$. This cannot be expected for $ \sH(\curl,\Omega)$ in itself because the kernel, $\ker(\curl)\subseteq \Lp{2}(\Omega)^3$, is an infinite-dimensional closed subspace. Thus, some (derivative) control on this kernel is needed. Before we present the nowadays classical selection theorem, we recall the complex property valid for the ranges and kernels of the operators just mentioned. The respective proofs can be found in many places; \cite[Chapter 6]{STW22} is one of them.

\begin{proposition}\label{prop:compl}
  Let $\Omega\subseteq \R^3$ be open. Then,
  \[
    \ran(\gradz)\subseteq \ker(\curlz),\quad \ran(\curlz)\subseteq \ker(\divz)
  \]
  and
  \[
    \ran(\grad)\subseteq \ker(\curl),\quad \ran(\curl)\subseteq \ker(\div).
  \]
\end{proposition}

With this, one readily confirms that
\[\ker(\curl)^\perp=\overline{\ran(\curl^*)}=\overline{\ran(\curlz)}\subseteq \ker(\divz).\]
Hence, the next result shows that intersecting $\sH(\curl,\Omega)$ with $\sH(\divz,\Omega)$ produces a convenient compactness result for $\curl$ outside its kernel and, at the same time, it provides the desired derivative control on the kernel.
\begin{theorem}[Picard--Weber--Weck selection theorem, \cite{Picard1984}]\label{thm:PWWst}
  Let $\Omega\subseteq \R^3$ be an open, bounded and (weak) Lipschitz (i.e., $\partial\Omega$ is a Lipschitz manifold) domain. Then,
  \[
    \sH(\divz,\Omega)\cap \sH(\curl,\Omega) \cpt \Lp{2}(\Omega)^{3}
    \quad\text{and}\quad
    \sH(\div,\Omega)\cap \sH(\curlz,\Omega) \cpt \Lp{2}(\Omega)^{3},
  \]
  where $\cpt$ denotes a compact embedding.
\end{theorem}

A mere standard consequence of the compact embedding result is the closedness of the ranges of $\curl$ (and hence of $\curl^*=\curlz$) and $\divz$ (and hence of $\divz^*=-\grad$). The abstract result reads as follows.

\begin{theorem}\label{thm:crchrcomp}
  Let $H_1,H_2$ be Hilbert spaces, $C\colon \dom(C)\subseteq H_1\to H_2$ closed and densely defined. 
  Then, the following conditions are equivalent:
  \begin{enumerate}
    \item\label{item:closed-range} $\ran(C)\subseteq H_2$ is closed.
    \item\label{item:bounded-from-below} $\exists c>0\, \forall x\in \dom(C)\cap \ker(C)^{\perp} \cond \norm{x}_{H_1}\leq c\norm{Cx}_{H_2}$.
    \item\label{item:adjoint-closed-range} $\ran(C^*)\subseteq H_1$ is closed.
  \end{enumerate}
  If $\dom(C)\cap \ker(C)^\perp\cpt H_1$, then \ref{item:bounded-from-below} holds. 
\end{theorem}

\begin{proof}
  That \ref{item:closed-range} is equivalent to \ref{item:bounded-from-below} is a standard application of the closed graph theorem, see, e.g., \cite[Theorem IV.1.6]{G06}. The equivalence of \ref{item:closed-range} and \ref{item:adjoint-closed-range} is Banach's closed range theorem (\cite[Theorem IV.1.2]{G06}). The final statement follows with a (standard) contradiction argument, see, e.g., \cite{DIW24} or again the FA-toolbox in \cite{PZ2020}.
\end{proof}
The closed range result in \Cref{thm:crchrcomp} asserts a certain continuous invertibility statement of $C$ apart from $\ker(C)$ (where such a result necessarily fails) and with codomain $\ran(C)$ (where such a result can only hold in the first place).
\begin{corollary}\label{cor:invform}
  Let $H_1,H_2$ be Hilbert spaces, $C\colon \dom(C)\subseteq H_1\to H_2$ closed and densely defined.
  If $\ran(C)\subseteq H_2$ is closed, then with $\iota_0\colon \ran(C)\hookrightarrow H_2$ and $\iota_1\colon \ran(C^\ast) \hookrightarrow H_1$
  \[
    \iota_0^*C\iota_1 \colon \dom(C)\cap \ran(C^\ast)\subseteq \ran(C^\ast)  \to \ran(C); u\mapsto Cu
  \]
  is continuously invertible.
\end{corollary}

\begin{proof}
\Cref{thm:crchrcomp} yields closedness of $\ran(C^\ast)$ and hence $\ran(C^\ast)=\ker(C)^\perp$.
  Thus, the operator is evidently one-to-one and onto. The continuous invertibility now follows from \ref{item:bounded-from-below} in \Cref{thm:crchrcomp}.
\end{proof}

 In \cite{TW14} this has been analysed in the context of divergence form problems; see also \cite[Section~1.2]{BW26}. In fact, the formulation in \cite[Theorem 2.2]{W22_tnt} led to the present formulation, which we shall make use of in the context of the div-curl system in question. Note that we require a slight refinement of the respective statement as we need some quantitative control in terms of the bounds of the coefficients. For $C\colon \dom(C)\subseteq H_1\to H_2$ closed and densely defined, $H_1,H_2$ Hilbert spaces, the subspace
\[
  V_C \coloneq \dom(C)\cap \ker(C)^\perp
\]
is considered a Hilbert space endowed with the respective graph scalar product.

\begin{theorem}\label{thm:wptw}
  Let $H_1,H_2$ be Hilbert spaces, $C\colon \dom(C)\subseteq H_1\to H_2$ closed and densely defined with $\ran(C)\subseteq H_2$ closed. Let $a\in \Lb(H_2)$ satisfy
  \[\Re a \coloneq (a+a^*)/2 \geq c\]
  for some $c>0$ in the sense of positive definiteness. Then, for every $f\in \ran(C^*)$, there exists a unique $u\in V_C$ such that
  \[
    \forall \phi \in V_C \cond \scprod{a C u}{C \phi}_{H_{2}} = \scprod{f}{\phi}_{H_{1}}.
  \]
  Moreover, with the canonical embeddings $\iota_0\colon \ran(C)\hookrightarrow H_2$ and $\iota_1\colon \ran(C^\ast) \hookrightarrow H_1$ we have
  \[
    u = (\iota_1^* C^*a C\iota_1)^{-1} f = (\iota_0^*C\iota_1)^{-1} (\iota_0^* a\iota_0)^{-1} (\iota_1^* C^*\iota_0)^{-1} f
  \]as well as
  \begin{align*}
    \norm{u}_{H_1}
    &\leq
      \norm{(\iota_0^*C\iota_1)^{-1}}_{\Lb(\ran(C),\ran(C^*))} \norm{(\iota_1^* C^*\iota_0)^{-1}}_{\Lb(\ran(C^*),\ran(C))} \frac{1}{c}  \norm{f}_{H_1}\text{ and} \\
    \norm{Cu}_{H_2}
    & \leq \sqrt{\norm{(\iota_0^*C\iota_1)^{-1}}_{\Lb(\ran(C),\ran(C^*))} \norm{(\iota_1^* C^*\iota_0)^{-1}}_{\Lb(\ran(C^*),\ran(C))}} \frac{1}{c} \norm{f}_{H_1}.
  \end{align*}
\end{theorem}

\begin{remark}\label{rem:distC}
  Unique existence for $u\in V_C$ also holds for given $f\in V_C'$ such that
  \[
    \forall \phi \in V_C \cond \scprod{a C u}{C \phi}_{H_2} = f(\phi).
  \]
  The respective result and proofs are contained, e.g., in \cite[Corollary~1.2.4]{BW26}, \cite[Theorem 2.9]{W18_NHC}, or also \cite[Theorem 3.1]{TW14}.
\end{remark}

\begin{proof}[Proof of \Cref{thm:wptw}]
  The results are in principle contained in the sources mentioned in \Cref{rem:distC}. More specifically, the reformulations in \cite[Theorem 2.2]{W22_tnt} provide the concrete solution operator representation at hand, i.e., with $f$ chosen from $\ran(C^*)$ instead of the whole $V_C'$. The continuous invertibility of $\iota_0^*C\iota_1$ (and its adjoint) follows from \Cref{cor:invform}. The continuity estimate for the $H_1$-norm of $u$ is then evident. For the second estimate, we use the equation satisfied by $u \in V_C$ and $\Re a\geq c$ as well as the $H_1$-estimate to obtain
  \begin{align*}
    \MoveEqLeft
    c \norm{Cu}_{H_2}^2 \\
    & \leq \Re \scprod{a Cu}{Cu}_{H_{2}} = \Re \scprod{f}{u}_{H_{1}} 
    \\
    & \leq \norm{f}_{H_1} \frac{1}{c} \norm{(\iota_0^*C\iota_1)^{-1}}_{\Lb(\ran(C),\ran(C^*))} \norm{(\iota_1^* C^*\iota_0)^{-1}}_{\Lb(\ran(C^*),\ran(C))} \norm{f}_{H_1},
  \end{align*}
  which establishes the claim.
\end{proof} 

Facilitating this result, we can provide a convenient solution theory for a linear version of the nonlinear equation in the focus of the present article. For this, let $0<\alpha\leq \beta$ and define
\[
  \mathcal{M}(\alpha,\beta) \coloneq \dset{a\in  \Lb(\Lp{2}(\Omega)^3)}{\Re a\geq \alpha, \Re a^{-1}\geq 1/\beta}.
\]

\begin{theorem}\label{thm:sotheorcurldiv}
  Let $\Omega\subseteq \R^3$ be an open and bounded weak Lipschitz domain, and let $0<\alpha\leq \beta$. Further, assume that $\Omega$ has a connected complement. Then, for all $f \in \Lp{2}(\Omega)$, $g\in \ran(\curlz)$, and $\varepsilon \in \mathcal{M}(\alpha,\beta)$, there is a unique $E\in \dom(\curlz)\cap \dom(\div\varepsilon)$ such that
  \[
    \div\varepsilon E = f \quad\text{and}\quad \curlz E = g.
  \]
  Moreover, there exists some $c\geq 0$ such that, for all $f \in \Lp{2}(\Omega)$, $g\in \ran(\mathring{\curl})$, and $\varepsilon \in \mathcal{M}(\alpha,\beta)$, we have
  \[
    \norm{E}_{\Lp{2}(\Omega)^3} \leq c( \norm{f}_{\Lp{2}(\Omega)^3} +\norm{g}_{\Lp{2}(\Omega)}),
  \]
  i.e., $c$ only depends on $\Omega$, $\alpha$, and $\beta$.
\end{theorem}

\begin{proof} 
  The proof essentially follows from \cite[Theorem 4.3]{W22_tnt}. For this, note that $\ran(\curl)$ is closed by \Cref{thm:PWWst} and \Cref{thm:crchrcomp} and $\ran(\gradz)$ is closed by Rellich's selection theorem. Next, since $\Omega$ has a connected complement, $\ker(\curlz)\cap \ker(\div)=\{0\}$ (by, e.g., \cite[Theorem 6.6]{PW22} or \cite[Theorem 1]{Pi1982}) and, together with \Cref{prop:compl}, we obtain the (orthogonal) Helmholtz decomposition
  \begin{align*}
      \Lp{2}(\Omega)^3 &= \ran(\gradz)\oplus \ker(\div)\\
      &= \ran(\gradz) \oplus \bigl(\ker(\div)\cap\ran(\curl) \bigr)\oplus \bigl(\ker(\div)\cap \ker(\curlz)\bigr)\\
      &= \ran(\gradz)\oplus \ran(\curl).
      \end{align*}
   Hence, by  \cite[Remark 3.4 and Corollary 3.3]{W22_tnt}, it follows that $\ker(\mathring{\curl})\cap \ker(\div\varepsilon)=\{0\}$, and we obtain the (orthogonal)
  generalised Helmholtz decomposition
  \[
    \Lp{2}(\Omega)^3 = \ran(\gradz) \oplus \varepsilon^{-1}\ran(\curl)
  \]
  by \cite[Theorem 3.2]{W22_tnt}.
  Thus, based on $\ker(\gradz)=\{0\}$, $\ran(\div)=\Lp{2}(\Omega)$, and \Cref{thm:wptw}, we easily see that
  \[
    E = \gradz (\div\varepsilon \gradz)^{-1} f + \varepsilon^{-1} \curl\iota_1 (\iota_1^*\curlz\varepsilon^{-1} \curl\iota_1)^{-1} g,
  \]
  where $\iota_1\colon \ker(\curl)^\perp\hookrightarrow \Lp{2}(
  \Omega)^3$ is the canonical embedding, is the unique solution with the corresponding estimate (for the estimate, also note that $\norm{\varepsilon^{-1}} \leq 1/\alpha$ as $\Re \varepsilon \geq \alpha$).
\end{proof}

\section{Homogenisation of Nonlocal Coefficients}\label{sec:homnonl}

In this section, we gather the necessary material from the theory of nonlocal homogenisation in the light of operator theory and $\Htopo$-convergence. Throughout this section, we let $\Omega\subseteq \R^3$ be an open, bounded, and weak Lipschitz domain and additionally assume that $\Omega$ has a connected complement. First, we quickly recall the consequences for the vector-analytic operators from the previous section in this setting. We have
\begin{enumerate}[label=\textup{(\alph{*})}]
  \item\label{item:PWWst} $\sH(\divz,\Omega)\cap \sH(\curl,\Omega) \cpt \Lp{2}(\Omega)^{3}$ (\Cref{thm:PWWst}),
  \item\label{item:closed-ranges-of-grad-and-curl} $\mathfrak{c}\coloneq \ran(\curl)$ and $\mathfrak{g}_0\coloneq \ran(\gradz)$ are closed (\Cref{thm:PWWst} together with \Cref{thm:crchrcomp}),
  \item\label{item:cohomology-group-trivial} $\ker(\div)\cap \ker(\curlz)=\{0\}$ (\cite[Theorem 6.6]{PW22} or \cite[Theorem 1]{Pi1982}), and
  \item\label{item:Helmholtz-decomposition} the Helmholtz decomposition $\Lp{2}(\Omega)^3 = \mathfrak{g}_0\oplus \mathfrak{c}$ (see \ref{item:closed-ranges-of-grad-and-curl}, \ref{item:cohomology-group-trivial}, and \Cref{prop:compl}).
\end{enumerate}
Note that, in particular, using \ref{item:closed-ranges-of-grad-and-curl}, we may apply \Cref{thm:wptw} (and \Cref{rem:distC}) to $C=\gradz$ or $C=\curl$. This is the basis of nonlocal $\Htopo$-convergence (or, more generally, the Schur topology as introduced in \cite{W18_NHC}). Focusing on the more particular situation here, we use the following ad-hoc introduction. For $0<\alpha\leq \beta$, we recall from the previous section that
\[
  \mathcal{M}(\alpha,\beta) = \dset[\big]{a\in \Lb(\Lp{2}(\Omega)^3)}{\Re a\geq \alpha, \Re a^{-1}\geq 1/\beta},
\]
as well as for closed $C\colon \dom(C)\subseteq H_1\to H_2$, the Hilbert space
\[
  V_C =\dom(C)\cap \ker(C)^{\perp}\text{ endowed with the graph scalar product of $C$}.
\] 
We say that $(a_n)_n$ in $  \mathcal{M}(\alpha,\beta)$ \textbf{$\Htopo$-nonlocally} converges to some $a \in   \mathcal{M}(\alpha,\beta)$, (i.e., in the \textbf{Schur topology} $\tau(\mathfrak{g}_0,\mathfrak{c})$), if the following condition holds:
\begin{itemize}[beginpenalty=10000,topsep=0.5\baselineskip]
\item
For all $f\in V_{\gradz}' (= \sH^{-1}(\Omega))$, $g\in V_{\curl}'$ and the corresponding sequences of unique solutions $u_n\in V_{\gradz} (=\sH^1_0(\Omega))$ and $v_n \in V_{\curl}$ to
\begin{align*}
 \forall \phi\in V_{\gradz} &\cond \scprod{a_n \gradz u_n}{\gradz \phi} = f(\phi)\text{, and} \\
 \forall \psi\in V_{\curl} &\cond \scprod{a_n^{-1} \curl v_n}{\curl \psi} = g(\psi)
\end{align*}
respectively, we have
\begin{align*}
 u_n \weakto u \in V_{\gradz}, &\quad v_n \weakto  v \in V_{\curl},\\
  a_n\gradz u_n \weakto  a\gradz u , &\quad    a_n^{-1}\curl v_n \weakto  a^{-1}\curl v \in \Lp{2}(\Omega)^3,
\end{align*}
where $u\in V_{\gradz}$ and $v\in V_{\curl}$ are the unique solutions to
\begin{align*}
 \forall \phi\in V_{\gradz} &\cond \scprod{a \gradz u}{\gradz \phi} = f(\phi)\text{, and} \\
 \forall \psi\in V_{\curl} &\cond \scprod{a^{-1} \curl v}{\curl \psi} = g(\psi),
\end{align*}
respectively.
\end{itemize}

The first insight required, when the Schur topology is concerned, is the following compactness result.

\begin{theorem}[{{\cite[Theorem 5.5, Theorem 5.10]{W18_NHC}}}]\label{thm:schurcomp} Let $0<\alpha\leq \beta$. Then,
\[
   (\mathcal{M}(\alpha,\beta), \tau(\mathfrak{g}_0,\mathfrak{c}))
\]
is metrisable and (sequentially) compact.
\end{theorem}
\begin{proof}
In \cite[Theorem 5.5, Theorem 5.10]{W18_NHC}, one finds metrisability of the topology and relative compactness of the set under consideration, i.e., there exists a subset of $\Lb(\Lp{2}(\Omega)^3)$ that contains $\mathcal{M}(\alpha,\beta)$ and, if endowed with the Schur topology, is metrisable and compact.

Let $(a_n)_n$ from $\mathcal{M}(\alpha,\beta)$ converge to some $a\in \Lb(\Lp{2}(\Omega)^3)$ with respect to $\tau(\mathfrak{g}_0,\mathfrak{c})$. Define
\[\tilde{a}_n\coloneq\begin{pmatrix}
  1/(2\alpha)+1/(2\beta)  &0\\
    0&a_n^{-1}
\end{pmatrix} \in \Lb(\Lp{2}(\Omega)^4)\]
for $n\in\N$ and $\tilde{a}$ similarly. Set $A\coloneq\begin{psmallmatrix}
    0&\div\\
    \gradz&0
\end{psmallmatrix}$. Then, $A$ is skew-selfadjoint and, by Rellich's selection theorem and a duality argument, $\dom(A)\cap\ker(A)^\perp$ as a Hilbert space is compactly embedded into $\Lp{2}(\Omega)^4$. Furthermore, by \cite[Lemma~5.4]{W22_tnt} and \cite[Theorem~4.10]{W18_NHC}, $a_n\to a$ with respect to $\tau(\mathfrak{g}_0,\mathfrak{c})$ implies $\tilde{a}_n\to\tilde{a}$ with respect to $\tau(\ker(A),\ran(A))=\tau(\{0\}\oplus \mathfrak{c},\Lp{2}(\Omega)\oplus \mathfrak{g}_0)$ from \cite[Theorem 3.2.3]{BW26}. This yields the closedness of the set.
\end{proof}

The present definition has a direct impact on the linear variants of the div-curl problems considered here.

\begin{theorem}[{{\cite[Theorem 7.5]{W22_tnt}}}]\label{thm:homnonl}
  Let $0<\alpha\leq \beta$, $(\varepsilon_n)_n$ in $\mathcal{M}(\alpha,\beta)$, and let $\varepsilon \in \mathcal{M}(\alpha,\beta)$ with $\varepsilon_n\to\varepsilon$ $\Htopo$-nonlocally. Then, if for $f\in \Lp{2}(\Omega)$ and $g\in \ran(\curlz)$,  $(E_n)_n$ in $\dom(\curlz)\cap \dom(\div\varepsilon)$ satisfies
  \[
    -\div (\varepsilon_n E_n) = f  \quad\text{and}\quad \curlz E_n = g ,
  \]
  then $E_n \weakto E$ in $\Lp{2}(\Omega)^3$, where $E\in \dom(\curlz)\cap \dom(\div\varepsilon)$ uniquely solves
  \[
    -\div (\varepsilon E ) = f \quad\text{and}\quad \curlz E  = g.
  \]
\end{theorem}

One of the more concrete applications of the developed theory are convolution type operators. As those are also relevant for our nonlinear system, we provide this example more explicitly. Moreover, in order to illustrate convergence in the Schur topology, we recall \cite[Example 6.7]{W18_NHC} in \ref{item:nonlocal-Hconvergence-convolution} of \Cref{ex:conv}.

\begin{example}\label{ex:conv}
\phantom{a}
  \begin{enumerate}[label=\textup{(\alph{*})}, beginpenalty=10000]
    \item\label{item:convolution-bounded-by-L1-norm}%
          Let $\phi\in \Lp{1}(\R^3)$. Then, for $d\in\N$ and $f\in  \Lp{2}(\Omega)^d$, define
          \[
          \phi \cvl f (x) \coloneq \int_{\Omega} \phi(x-y) f(y) \dx[y].
          \]
          By Young's inequality, we infer $\phi\cvl\in \Lb(\Lp{2}(\Omega)^d)$ with
          \[
          \norm{\phi\cvl}_{\Lb(\Lp{2}(\Omega)^d)}\leq \norm{\phi}_{\Lp{1}(\R^3)}.
          \]
    \item\label{item:nonlocal-Hconvergence-convolution}%
          For $0<\alpha\leq \beta$, let $(a_n)_n$ be any nonlocally $\Htopo$-convergent sequence in $\cM(\alpha,\beta)$ with limit $a\in \cM(\alpha,\beta)$, let $(k_n)_n$ be a bounded sequence in $\soboW_{\infty}^{1}(\R^3)$, and define
          \[
          k_n \ast \phi (x)\coloneq \int_{\Omega} k_n(x-y)\phi(y) \dx[y]
          \]
          for $\phi\in \Lp{2}(\Omega)^3$.
          If $k_n\to k\in \Lp{\infty}(\R^{3})$ in $\sigma(\Lp{\infty}\Lp{1})$,
          and if there exists $c>0$ with $a_n + k_n\ast\mathopen{}=(a_n + k_n\ast\mathopen{})^\ast\geq c$ for $n\in\N$, then, by \cite[Example 6.7]{W18_NHC}, there exist $0<\alpha^\prime\leq \beta^\prime$ such that $a_n + k_n\ast\mathopen{}\in \cM(\alpha^\prime,\beta^\prime)$ for $n\in\N$ and $(a_n + k_n\ast\mathopen{})_n$ nonlocally $\Htopo$-converges to $a+k \ast\mathopen{}$.
  \end{enumerate}
\end{example}

In the proof of our main result, we will need the following characterisation of nonlocal $\Htopo$-convergence. In fact, it is this very characterisation that was used in \cite[Example 6.7]{W18_NHC} to prove the statement of \Cref{ex:conv} \ref{item:nonlocal-Hconvergence-convolution}.

For Hilbert spaces $H_1,H_2$ and $C\colon \dom(C)\subseteq H_1\to H_2$ closed and densely defined, we introduce the operator
\[
  C^\diamond \colon
  \begin{cases}
\hfill H_2 &\to\quad V_C^\prime\\
\hfill q&\mapsto\quad \bigl[\phi\mapsto \scprod{q}{C\phi}_{H_2}\bigr]
\end{cases}.
\]
Note that $C^\diamond$ is a natural extension of $C^*$ (see, e.g., \cite[Chapter 9]{STW22}). 
To distinguish between the adjoint operators and the ones mapping into the dual spaces, we will append a `$-1$' in the index. For example, we shall write
\[
  \div_{-1}\coloneq - \gradz^\diamond\colon \Lp{2}(\Omega)^3 \to \sH^{-1}(\Omega)
  \quad\text{and}\quad
  \curlz_{-1}\coloneq\curl^\diamond \colon \Lp{2}(\Omega)^3 \to V_{\curl}'
\]
in the following.
\begin{theorem}[{{\cite[Theorem 6.5]{W18_NHC}}}]\label{thm:dccrit} Let $(a_n)_n$ and $a$ in $\cM(\alpha,\beta)$. Then, the following conditions are equivalent:
  \begin{enumerate}
    \item $a_n\to a$ in $\tau(\mathfrak{g}_0,\mathfrak{c})$.

    \item for all $(q_n)_n$ in $\Lp{2}(\Omega)^3$ weakly convergent to $q \in \Lp{2}(\Omega)^3$ and all $\kappa\colon \N\to\N$ strictly monotone, the following implication holds: if
          \begin{equation*}
            \dset{\div_{-1} a_{\kappa(n)}q_{n}}{n\in \N} \subseteq \sH^{-1}(\Omega)
            \quad\text{and}\quad
            \dset{\curlz_{-1} q_{n}}{n\in \N}\subseteq V_{\curl}'
          \end{equation*}
          are relatively compact,
          then $a_{\kappa(n)}q_n \rightharpoonup aq$.
  \end{enumerate}
\end{theorem}

\section{Main Result}\label{sec:mr}

Throughout this section, we let $\Omega\subseteq \R^3$ be an open, bounded, and weak Lipschitz domain and additionally assume that $\Omega$ has a connected complement. We make use of the vector-analytic operators introduced in \Cref{sec:background} and the definition of the convolution $\cvl$ provided in \Cref{ex:conv}.

\begin{theorem}\label{thm:mt}
Let $0<\alpha\leq \beta$, $a\in \mathcal{M}(\alpha,\beta)$, $F\colon \C^3 \to \C$ continuous and bounded, $f\in \Lp{2}(\Omega)^3$ and $g\in \ran(\curlz)$.

If $\norm{F}_{\infty}\lambda(\Omega)<\alpha$,\footnote{Here $\lambda$ denotes the Lebesgue measure on $\R^{3}$.} then there exists $E\in \Lp{2}(\Omega)^3$ with $\varepsilon(E)\in \sH(\div,\Omega)$ and $E\in \sH(\curlz,\Omega)$ such that
\[
  -\div \varepsilon(E) = f
  \quad\text{and}\quad
  \curlz E = g,
\]
where  $\varepsilon(E) \coloneq a E + F(E)*E$.
\end{theorem}

The existence result is based on Schauder's theorem (see \Cref{thm:Sch} below), which in turn requires a compactness result. For this compactness result, we need the interconnected effects of the weak operator topology (denoted by $\tau_{\textnormal{w}}$) and the Schur topology (i.e., nonlocal $\Htopo$-convergence, denoted by $\tau(\mathfrak{g}_0,\mathfrak{c})$).

\begin{theorem}[{{\cite[Theorem 9.1--Helga's theorem]{W22_tnt}}}]\label{thm:Helga}
Let $0<\alpha\leq \beta$ and $\varepsilon_n\in \cM(\alpha,\beta)$ for $n\in\N$.
Assume that there exists $\varepsilon\in \cM(\alpha,\beta)$ such that
\[
  \varepsilon_n \stackrel{\tau_{\textnormal{w}}}{\to}\varepsilon
  \quad\text{and}\quad
  \varepsilon_n \stackrel{\tau(\mathfrak{g}_0,\mathfrak{c})}{\to}\varepsilon.
\]
In addition, assume for all $G\in \Lp{2}(\Omega)^3$
\begin{equation*}
  \dset{\div_{-1} (\varepsilon_n G)}{n\in \N} \subseteq \soboH^{-1}(\Omega)
\end{equation*}
is relatively compact.

If $(E_n)_n$ is a sequence in $\sH(\curlz,\Omega)$ such that $(\varepsilon_n E_n)_n$ is in $\sH(\div,\Omega)$ and
\begin{enumerate}
  \item $(E_n)_n$ is bounded in $\Lp{2}(\Omega)^{3}$,

  \item $(\curlz (E_{n}))_{n}$ is bounded in $\Lp{2}(\Omega)^{3}$, and

  \item $(\div(\varepsilon_{n}E_{n}))_{n}$ is bounded in $\Lp{2}(\Omega)$,
\end{enumerate}
then $(E_{n})_{n}$ has an $\Lp{2}(\Omega)^{3}$-convergent subsequence.
\end{theorem}

\begin{remark}\label{rem:SOTimplWOTandSchur}
Note that in the latter compactness result, the limits in the weak operator topology on the one hand and in the Schur topology on the other hand need to coincide. Although both topological spaces are (relatively) compact (and, in the present setting, even metrisable) on bounded sets, the respective limits need not be the same. Hence, this coincidence of the limits is a necessary additional assumption. It is not difficult to see that convergence in the strong operator topology implies this coincidence of limits. However, it is often useful to have a stronger statement at hand (i.e., one with weaker assumptions). Finally, let us recall that the coincidence of homogenisation and weak operator topology limit is a useful assumption for the derivation of Hashin--Shtrikman bounds; see, e.g., \cite{Tartar2009}.
\end{remark}

\section{Proof of the Main Result}\label{sec:pmr}

The proof is based on Schauder's fixed point theorem and several results, making it applicable. Before we turn to the proof of the existence result of present interest, \Cref{thm:mt}, we shall state and prove the respective auxiliary results first. 

\begin{theorem}[Schauder's fixed point theorem, \cite{Sch30}]\label{thm:Sch}
  Let $X$ be a Banach space, $\emptyset \neq M\subseteq X$ convex, closed,
  $T\colon M\to M$ continuous with $T[M] \subseteq K$ for a compact $K\subseteq M$.
  Then, $T$ has a fixed point.
\end{theorem}

\begin{lemma}\label{lem:fsop} Let $F\colon \C^3 \to \C$ be continuous and bounded, $\Omega\subseteq \R^3$ open and bounded. Further, let $(G_n)_n$ in $\Lp{2}(\Omega)^3$ converge to some $G\in \Lp{2}(\Omega)^3$. Then,
\[
   F(G_n)\cvl\to F(G)\cvl\text{ in the strong operator topology of }\Lb(\Lp{2}(\Omega)^3).
\]
\end{lemma}
\begin{proof}
 Let $(G_{n_k})_k$ be any subsequence of $(G_n)_n$. Then, we can choose a further subsequence which converges almost everywhere; thus $F(G_{n_k})\to F(G)$ almost everywhere. Hence, by Lebesgue's dominated convergence theorem, it follows that $F(G_{n_k})\cvl\phi \to F(G)\cvl\phi$ for all $\phi\in \Lp{2}(\Omega)^3$ as $k\to\infty$. The subsubsequence principle (see, e.g., \cite[Proposition 2.1.2]{BW26}) establishes the assertion.
\end{proof}

Throughout the remaining section, we let $\Omega\subseteq \R^3$ be an open, bounded, and weak Lipschitz domain and additionally assume that $\Omega$ has a connected complement.

\begin{theorem}\label{con:wnl}
  Let $0<\alpha\leq \beta$ and $\varepsilon_n\in \cM(\alpha,\beta)$ for $n\in\N$.
 Assume that there exists $\varepsilon\in \Lb(\Lp{2}(\Omega)^3)$ such that
  \begin{equation*}
    \varepsilon_n \stackrel{\tau_{\textnormal{w}}}{\to}\varepsilon
  \end{equation*}
  and that for all $G\in \Lp{2}(\Omega)^3$
  \begin{equation*}
    \dset{\div_{-1} (\varepsilon_n G)}{n\in \N} \subseteq \soboH^{-1}(\Omega)
  \end{equation*}
  is relatively compact.

  Then, $\varepsilon\in \cM(\alpha,\beta)$ and
  \begin{equation*}
    \varepsilon_n \stackrel{\tau(\mathfrak{g}_0,\mathfrak{c})}{\to}\varepsilon.
  \end{equation*}
\end{theorem}

\begin{proof}
  By \Cref{thm:schurcomp}, we find $\kappa\colon \N\to\N$ strictly monotone such that $(\varepsilon_{\kappa(n)})_n$  $\Htopo$-nonlocally converges to some $b\in \cM(\alpha,\beta)$. Let $q\in \Lp{2}(\Omega)^3$. By assumption 
  \[
    \dset{\div_{-1} \varepsilon_{\kappa(n)} q}{n\in \N} \subseteq \sH^{-1}(\Omega)
  \]
  is relatively compact. Since, trivially,
  \[
    \{\curlz_{-1}  q \} \subseteq V_{\curl}'
  \]
  is relatively compact too, we infer
  \[
    \varepsilon_{\kappa(n)}q \weakto bq
  \]
  by virtue of \Cref{thm:dccrit}. In particular, $\varepsilon_{\kappa(n)}\to b$ in the weak operator topology. Thus, $b=\varepsilon$ and $\kappa$ can be chosen to be the identity by applying the subsubsequence principle.
\end{proof}

\begin{proposition}\label{prop:convo}
  Let $\mathcal{B}\subseteq \Lp{2}(\Omega)$ be bounded. Then, for all $G\in \Lp{2}(\Omega)^{3}$, the set
  \begin{equation*}
    \dset{\div_{-1} (E*G)}{E\in \mathcal{B}} \subseteq \soboH^{-1}(\Omega)
  \end{equation*}
  is relatively compact.
\end{proposition}

\begin{proof} Without loss of generality we can assume $\mathcal{B} = \ball(0,1)$, i.e., the ball centered in $0$ with radius $1$ in $\Lp{2}(\Omega)$.
  \begin{steps}
    \item\label{item:compactness-proof-step-1} We additionally assume that $G\in \soboH(\div,\Omega)$. Then
    \[
      \mathcal{T}_G \coloneq \dset{ \div (E*G)}{E\in \ball(0,1)} = \dset{E*\div G}{E\in \ball(0,1)} \subseteq \Lp{2}(\Omega).
    \]
    Since $\ball(0,1)$ is bounded in $ \Lp{2}(\Omega)$ and, thus, in $ \Lp{1}(\Omega)$ Young's inequality implies that $\mathcal{T}_G$ is bounded in $\Lp{2}(\Omega)$.
    By the compactness of $ \Lp{2}(\Omega) \hookrightarrow \soboH^{-1}(\Omega)$, i.e., Rellich's selection theorem, $\mathcal{T}_G\subseteq \sH^{-1}(\Omega)$ is compact.

    \item
    Consider some general $G\in \Lp{2}(\Omega)^3$.
    By the continuity of $\div_{-1} \colon \Lp{2}(\Omega)^3 \to \soboH^{-1}(\Omega)$ and the density of $\soboH(\div,\Omega)$ in $\Lp{2}(\Omega)$, for every $\epsilon > 0$,
    we find an $H \in \soboH(\div,\Omega)$ such that
    \begin{equation}\label{eq:approxL2withHdiv}
        \norm{\div_{-1} G  - \div H}_{\soboH^{-1}(\Omega)}\leq \varepsilon/(2\sqrt{\lambda(\Omega)}).
    \end{equation}
     Thus, for all $E\in B$ and $\phi\in \sH^1_0(\Omega)$, we have
    \begin{align*}
    \MoveEqLeft
      \scprod{\div (E \cvl G) -\div (E\cvl H)}{\phi}_{\soboH^{-1}(\Omega)}\\
      &= -\scprod{E\cvl G-E\cvl H}{\gradz \phi}_{\Lp{2}(\Omega)^3} = -\scprod{G-H}{(E\cvl )^*\gradz\phi}_{\Lp{2}(\Omega)^3} \\
      &=-\scprod{G-H}{\grad((E\cvl )^* \phi)}_{\Lp{2}(\Omega)^3} = \scprod{\div_{-1}( G-H)}{(E\cvl )^* \phi}_{\soboH^{-1}(\Omega)}.
    \end{align*}
    Hence, using~\eqref{eq:approxL2withHdiv} and \Cref{ex:conv}~\ref{item:convolution-bounded-by-L1-norm} gives
    \begin{multline*}
      \sup_{\phi\in H_0^1(\Omega), \norm{\phi} \leq 1} \abs{\scprod{\div E\cvl G -\div E\cvl H}{\phi}_{\soboH^{-1}(\Omega)}} \\
      \leq \sup_{\phi\in H_0^1(\Omega), \norm{\phi} \leq 1}\frac{\varepsilon}{2\sqrt{\lambda(\Omega)}} \norm{E}_{\Lp{1}(\Omega)} \norm{\phi}_{\soboH_0^1(\Omega)} \leq \varepsilon/2.
    \end{multline*}
    Since $\mathcal{T}_H$ is relatively compact by \ref{item:compactness-proof-step-1}, we find a finite $\mathcal{F}\subseteq \Lp{2}(\Omega)$ such that
    \[
      \mathcal{T}_H \subseteq \bigcup_{E\in \mathcal{F}} \ball(E,\varepsilon/2).
    \]
    Hence,
    \[
      \mathcal{T}_G \subseteq \bigcup_{E\in \mathcal{F}} \ball(E,\varepsilon)
    \]
    shows that $\mathcal{T}_{G}$ is totally bounded and therefore relatively compact.\qedhere
  \end{steps}
\end{proof}

\begin{proof}[Proof of \Cref{thm:mt}]
  First of all, note that for $G\in \Lp{2}(\Omega)^3$, we have by H{\"o}lder's inequality
  \[
    \norm{F(G)}_{\Lp{1}(\Omega)}\leq \lambda(\Omega)\norm{F}_{\infty}.
  \]
  Hence, $\norm{F(G)\cvl}_{\Lb(\Lp{2}(\Omega)^3)}\leq \lambda(\Omega)\norm{F}_{\infty}$ by \Cref{ex:conv}. Therefore,
  \[
    \Re (a+F(G)\cvl) \geq \alpha - \lambda(\Omega)\norm{F}_{\infty}>0
  \]
and subsequently
\[
\Re (a+F(G)\cvl)^{-1}\geq \frac{\alpha - \lambda(\Omega)\norm{F}_{\infty}}{\norm{a+F(G)\cvl}^2_{\Lb(\Lp{2}(\Omega)^3)}}\geq \frac{\alpha - \lambda(\Omega)\norm{F}_{\infty}}{(\beta +\lambda(\Omega)\norm{F}_{\infty})^2}>0.
\]
  
  Thus, \Cref{thm:sotheorcurldiv} yields that
  $\Phi\colon \Lp{2}(\Omega)^3 \to \Lp{2}(\Omega)^3$ defined via $\Phi(G)\coloneqq E$, where $E\in \dom(\curlz)\cap \dom(\div (a + F(G)\cvl))$ uniquely solves
  \[
    -\div ( a E + F(G)\cvl E)= f\quad\text{and}\quad \curlz E = g
  \]
  is well-defined. More particularly, by virtue of \Cref{thm:sotheorcurldiv}, there exists $\kappa>0$ such that $\tilde{M}\coloneq \Phi[\Lp{2}(\Omega)^3]\subseteq \ball(0,\kappa)\subseteq \Lp{2}(\Omega)^3$.
  
  We claim that $\tilde{M}$ is relatively compact. For this, let $(E_n)_n=(\Phi(G_n))_n$ in $\tilde{M}$. By the boundedness of $F$ and $\Omega$, we find a weakly convergent subsequence $(F(G_{n_k}))_k$ in $\Lp{2}(\Omega)$ converging to some $H\in\Lp{2}(\Omega)$. It is not difficult to see that $F(G_{n_k})\cvl \to H\cvl $ in the weak operator topology. By \Cref{prop:convo}, for all $\phi\in \Lp{2}(\Omega)^3$,
  \[
    \dset{\div_{-1} (F(G_{n_k})*\phi)}{k\in \N}\subseteq \soboH^{-1}(\Omega)
  \]
  is relatively compact and so, by \Cref{con:wnl}, $a + F(G_{n_k})\cvl \to a+ H\cvl $ in the nonlocal $\Htopo$-topology.
  Therefore, \Cref{thm:homnonl} implies $E_{n_k}\weakto E$, where $E\in \dom(\curlz)\cap \dom(\div (a + H\cvl))$ uniquely solves
  \[
    -\div ( a E + H\cvl E)= f\quad\text{and}\quad \curlz E = g.
  \]
   Now, \Cref{thm:Helga} leads to $E_{n_k}\to E$ in $\Lp{2}(\Omega)^3$. Hence, $\tilde{M}$ is relatively compact.

  By \cite[Theorem~11.4]{V20}, $M\coloneq \mathop{\overline{\conv}} \tilde{M}$ (the closure of the convex hull of $\tilde{M}$) is totally bounded (precompact) and complete, hence compact.
  Thus, consider $T\coloneq \Phi|_M$, which evidently maps $M$ into $M$.
  Finally, let $(G_n)_n$ be a convergent sequence in $M$. By \Cref{lem:fsop}, $F(G_n)\cvl\to F(G)\cvl$ in the strong operator topology, and, similarly to the previous arguments, we obtain $\Phi(G_n)\to \Phi(G)$ in $M$. Hence, $T$ is continuous and, by \Cref{thm:Sch}, $T$ has a fixed point.
\end{proof}

\end{document}